\documentclass[11pt,reqno]{amsart}

\usepackage{amsmath,amssymb,amsfonts,bm}
\usepackage{stmaryrd}
\usepackage{bbold}
\usepackage{times}
\usepackage[all]{xy}
\usepackage{mathrsfs}
\usepackage{graphicx}
\usepackage{amsthm}

\theoremstyle{plain}
\newtheorem{thm}{Theorem}
\newtheorem{lem}[thm]{Lemma}

\newtheorem{prop}[thm]{Proposition}

\theoremstyle{definition}

\newtheorem*{ack}{Acknowledgment}

\theoremstyle{remark}

\renewcommand{\Im}{\operatorname{Im}}
\newcommand{\Hom}{\operatorname{Hom}}
\newcommand{\End}{\operatorname{End}}
\newcommand{\Aut}{\operatorname{Aut}}

\newcommand{\Ind}{\operatorname{Ind}}

\newcommand{\Res}{\operatorname{Res}}
\newcommand{\sS}{\mathcal{S}}

\newcommand{\edim}{\mathfrak{d}}

\newcommand{\Spec}{\operatorname{Spec}}

\newcommand{\Id}{\operatorname{Id}}
\newcommand{\rank}{\operatorname{rank}}

\newcommand{\diag}{\operatorname{diag}}
\newcommand{\tr}{\operatorname{Tr}}
\newcommand{\trace}{\operatorname{trace}}

\newcommand{\ch}{\operatorname{char}}

\newcommand{\onto}{\twoheadrightarrow}

\newcommand{\into}{\hookrightarrow}
\newcommand{\tto}{\longrightarrow}
\newcommand{\iso}{\stackrel{\sim}{\tto}}

\newcommand{\til}[1]{\widetilde{#1}}

\newcommand{\Mat}{\operatorname{M}}

\DeclareMathOperator{\rad}{rad}
\DeclareMathOperator{\soc}{soc}

\renewcommand{\phi}{\varphi}

\newcommand{\bl}{\boldsymbol{l}}

\renewcommand{\d}{\delta}

\newcommand{\fp}{\mathfrak{p}}

\newcommand{\e}{\varepsilon}

\renewcommand{\k}{\mathbb{k}}

\newcommand{\ZZ}{\mathbb{Z}}

\newcommand{\cat}[1]{\operatorname{\mathsf{#1}}}
\newcommand{\Mod}[1]{{#1}\text{-}\!\cat{Mod}}

\newcommand{\proj}[1]{{#1}\text{-}\!\cat{proj}}
\newcommand{\fgmod}[1]{{#1}\text{-}\!\cat{mod}}

\newcommand{\FPmod}[1]{{#1}\text{-}\!\cat{FPmod}_\infty}

\newcommand{\cen}{\mathcal{Z}}

\newcommand{\dC}{\delta_\text{\rm C}}

\newcommand{\Hint}{\textstyle\int}

\newcommand{\upin}{\text{\ \rotatebox{90}{$\in$}\ }}
\newdir{ >}{{}*!/-5pt/\dir{>}}

\renewcommand{\labelenumi}{(\alph{enumi})}

\begin{document}

\title[Projective modules]%
{Projective modules over Frobenius algebras and Hopf comodule algebras}

\author{Martin Lorenz}

\address{Department of Mathematics, Temple University,
    Philadelphia, PA 19122}

\email{lorenz@temple.edu}

\urladdr{http://www.math.temple.edu/$\stackrel{\sim}{\phantom{.}}$lorenz}

\author{Loretta FitzGerald Tokoly}

\address{Mathematics Division,
Howard Community College,
Columbia, MD 21044
}

\email{ltokoly@howardcc.edu}


\dedicatory{Dedicated to Mia Cohen on the occasion of her retirement}

\thanks{Research of the first author supported in part by NSA Grant H98230-09-1-0026}



\begin{abstract}
This note presents some results on projective modules and the
Grothendieck groups $K_0$ and $G_0$ for Frobenius algebras and for
certain Hopf Galois extensions. Our principal technical tools are
the Higman trace for Frobenius algebras and a product
formula for Hattori-Stallings ranks of projectives over 
Hopf Galois extensions.
\end{abstract}

\maketitle


\section*{Introduction}

The aim of this article is to generalize certain results from
\cite{mL97a} on projective modules and
the Grothendieck groups $K_0$ and $G_0$ for finite-dimensional Hopf algebras.
Here, we take a more ring theoretic approach and consider general Frobenius 
algebras and Hopf Galois extensions instead of finite-dimensional Hopf algebras. 
Moreover, for the most part, we work over a
commutative base ring rather than a field.

Section~\ref{S:ranks} serves to review some fairly standard material, notably the 
relationship between Hattori-Stallings ranks and ordinary characters of projective modules.
This relationship is stated in the context of the Grothendieck groups $G_0(A) = K_0(\fgmod{A})$ and 
$K_0(A) = K_0(\proj{A})$, where $\fgmod{A}$ denotes the category of all finitely generated left modules over the ring $A$ and $\proj{A}$ is the full subcategory consisting of all finitely generated 
projective left $A$-modules.
Bass \cite{hB76} and Brown \cite[IX.2]{kB82} are excellent background references for this section.

The core material of this article consists of Sections~\ref{S:Frobenius} and \ref{S:Hopf}
which are largely independent of each other. Section~\ref{S:Frobenius} deals with Frobenius
algebras $A$ over a commutative ring $\k$, the main theme being the \emph{Higman trace}
\[
\tau \colon A \tto A \ ;
\]
see \ref{SS:Higman}.
The main result of Section~\ref{S:Frobenius} is Theorem~\ref{T:Frobenius},
a generalization of \cite[Theorem 3.4]{mL97a}. It concerns the so-called \emph{Cartan map}
$c \colon K_0(A) \to G_0(A)$ coming from the inclusion $\proj{A} \into \fgmod{A}$.
Under the assumption that $\k$ is a splitting field for $A$, we show that the rank of the
$\k$-linear map
\[
c \otimes 1_\k \colon K_0(A) \otimes_{\ZZ} \k \tto G_0(A) \otimes_{\ZZ} \k 
\]
is identical to the rank of the Higman trace. 
With an additional technical hypothesis on the Higman trace, Theorem~\ref{T:Frobenius}
also states that invertibility
of $c \otimes 1_\k$ implies semisimplicity of $A$.

Section~\ref{S:Hopf} is based on some results from the second author's PhD thesis \cite{lT99}.
The main result of this section,
Theorem~\ref{T:Frob}, gives a condition on the possible ranks of finitely generated projectives over 
certain Hopf Galois extensions which generalizes  \cite[Theorem 2.3(b)]{mL97a}. 
The proof of Theorem~\ref{T:Frob} given here is different from the original one in \cite{lT99}, 
the essential new ingredient being a product formula for Hattori-Stallings ranks; see Lemma~\ref{L:product}.
This allowed for a more general version of the result.

\begin{ack}
The first author thanks the organizers of the conference on the occasion of Mia Cohen's retirement,
held at the Sde-Boker campus of Ben-Gurion University on May 24-27, 2010, for giving him the opportunity to
present some related results on ring theoretic methods in the representation theory of Hopf algebras to an expert audience.
\end{ack}


\section{Ranks and characters} \label{S:ranks}


\subsection{Hattori-Stallings ranks}  \label{SS:Hattori}

Let $A$ be any ring (associative with 1). We let $[A,A]$ denote the additive subgroup of $A$
that is generated by all Lie commutators $[x,y] = xy - yx$ with $x,y \in A$ and 
consider the canonical group epimorphism
\[
T \colon A \onto T(A) = A/[A,A]\ , \qquad a \mapsto T(a) = a + [A,A] \ .
\]
Now let $P$ be a finitely generated 
projective (left) $A$-module.
The \emph{trace map} is defined by
\begin{align*}
\tr_{P/A}: \End_A(P) \iso \Hom_A(P,A) &\otimes_A P
& &\tto &  T&(A) \\
&\upin &&& &\upin\\
f&\otimes v & &\longmapsto & T(&f(v))\ .
\end{align*}
If $\{(f_i, v_i)\}_1^n \subseteq\Hom_A(P,A)\times P$ are
dual bases for $P$, that is,  $v = \sum_i f_i(v)v_i$ holds for all $v\in P$,
then
\begin{equation} \label{E:Hattori}
\tr_{P/A}(\phi) = \sum_i T(f_i(\phi(v_i))) \qquad (\phi\in\End_A(P))\ .
\end{equation}
The \emph{Hattori-Stallings rank} of $P$ is defined by
\[
r(P) = r_A(P) = \tr_{P/A}(1_P) \in T(A)\ .
\]
Explicitly, writing $P \cong A^n e$ for some idempotent matrix
$e=(e_{i,j})\in M_n(A)$, we have 
\[
r(P) = \sum_i T(e_{i,i})\ .
\]
In particular,
if $P \cong A^n$ then $r_P = n\,T(1)$. 

Hattori-Stallings ranks are additive, that is, $r_{P \oplus Q} = r_P + r_Q$ holds for any 
two finitely generated projective $A$-modules $P$ and $Q$\,. Thus we obtain a group homomorphism
\[
r = r_A \colon K_0(A)\to T(A)\,,\qquad [P]\mapsto r(P)  \ .
\]

\subsubsection{Functoriality} \label{SSS:induction}
Given any ring homomorphism $f \colon A \to B$, the
canonical group homomorphisms $K_0(f) = \Ind_A^B \colon K_0(A) \to K_0(B)$, $[P] \mapsto [B \otimes_A P]$ (``induction''),
and $T(f) \colon T(A) \to T(B)$, $T(a) \mapsto T(f(a))$, fit into  a commutative diagram
\begin{equation} \label{E:Hattori2}
\xymatrix{%
K_0(A) \ar[r]^{K_0(f)} \ar[d]_{r_A} & K_0(B) \ar[d]^{r_B}\\
T(A) \ar[r]_{T(f)} & T(B)
} 
\end{equation}

\subsubsection{Commutative rings} \label{SSS:commutative}
For any commutative ring $A$, the Hattori-Stallings rank function $r \colon K_0(A) \to T(A) = A$
of Section~\ref{SS:Hattori} factors through the rank map
\begin{equation*}
\rank \colon K_0(A) \stackrel{}{\tto} H_0(A):= [\Spec A,\ZZ] \ ,
\end{equation*}
where $[\Spec A,\ZZ]$ denotes the collection of all continuous functions $\Spec A \to \ZZ$ with
$\ZZ$ carrying the discrete topology. 
For any $P$ in $\proj{A}$, the value of $\rank(P)$ on $\fp \in \Spec A$, denoted by $\rank_\fp(P)$,
is defined to be the ordinary rank of the free $A_\fp$-module $P_\fp = A_\fp \otimes_A P$:
\[
P_\fp \cong A_\fp^{\rank_\fp(P)}\ ;
\]
Any continuous function $f \colon \Spec A \to \ZZ$  has only finitely many values, 
because $\Spec A$ is quasi-compact. If $f$ has values $f_1,\dots,f_c$\,, say, then we can write
$A = \prod_{i=1}^c e_iA$ with orthogonal idempotents $e_i = e_i^2 \in A$
in such a way that the various $\Spec e_iA$ are exactly the fibres of $f$; 
see \cite[IX.3]{hB68} for all this.  Defining $H_0(A) \to A$ by $f \mapsto \sum f_ie_i$
it is easy to see that the Hattori-Stallings rank function factors as
\begin{equation*}
r \colon K_0(A) \stackrel{\rank }{\tto} H_0(A) \to A \ ;
\end{equation*}
see \cite[Chapter II]{cWxx}.

\subsubsection{Algebras} \label{SSS:sss}
If $A$ is an algebra over some commutative base ring $\k$ then $T(A)$ is a $\k$-module
and the homomorphism $r$ extends canonically to a $\k$-module map
\begin{equation} \label{E:rk}
r_\k  \colon K_0(A) \otimes_{\ZZ} \k \to T(A)\,,\qquad [P]  \otimes k \mapsto kr(P)  \ .
\end{equation}


\subsection{Characters}  \label{SS:characters}

Assume that $A$ is an algebra over some commutative base ring $\k$
and let $\fgmod{A}_\k$ denote the full subcategory
of $\fgmod{A}$ consisting of all $A$-modules that are finitely generated projective over $\k$.
The \emph{character} $\chi_V$ of a module  $V$ in $\fgmod{A}_\k$ is defined by
\[
\chi_V(a) = \tr_{V/\k}(a_V) \in \k \qquad (a \in A)\ ,
\]
where $a_V \in \End_\k(V)$ is given by $a_V(v) = av$\,. Thus,
\[
\chi_V \in T(A)^* \subseteq A^*\ ,
\]
where $\,.\,^* = \Hom_\k(\,.\,,\k)$ denotes the $\k$-linear dual and $T(A)^*$, the $\k$-module of all \emph{trace forms} on $A$, consists of the $\k$-linear forms $A \to \k$
that vanish on $[A,A]$. 
Following Swan 
\cite{rSeE70} we let $G_0^\k(A)$ denote the Grothendieck group of $\fgmod{A}_\k$\,.
Thus $G_0^\k(A)$ is the abelian group with generators $[V]$ for each module $V$ in $\fgmod{A}_\k$
and relations $[V] = [U] + [W]$ for each
exact sequence $0 \to U \to V \to W \to 0$ in $\fgmod{A}_\k$\,. 
Since we also have the relation $\chi_V = \chi_U + \chi_W$
in $T(A)^{*}$, we obtain a well-defined group homomorphism
\[
\chi \colon G_0^\k(A) \tto T(A)^{*} \ , \qquad [V] \mapsto \chi_V \ .
\]


\subsection{Characters of projectives}  \label{SS:projectives}

Assume that the algebra $A$ is finitely generated projective over $\k$\,.
Then each finitely generated projective $A$-module $P$ belongs to $\fgmod{A}_\k$\,,
and hence both $r(P) \in T(A)$ and $\chi_P \in T(A)^{*}$ are defined.
In fact, the Hattori-Stallings rank $r(P)$ determines the character $\chi_P$. For finite group algebras $A = \k G$ this was
spelled out explicitly by Hattori \cite{aH65}; see also \cite[5.8]{hB76}. The proposition below 
is taken from Bass \cite[4.7]{hB76}.

The inclusion 
$\proj{A} \into \fgmod{A}_\k$ gives rise to
a group homorphism 
\[
c^\k \colon K_0(A) \to G_0^\k(A) \ , \qquad [P] \mapsto [P] \ .
\]
If the base ring $\k$ is regular then the inclusion $\fgmod{A}_\k \into \fgmod{A}$ gives rise to
an isomorphism $G_0^\k(A) \iso G_0(A)$\,; see \cite[Theorem 1.2]{rSeE70}. Thus, identifying 
$G_0^\k(A)$ and $G_0(A)$ for regular $\k$, the map $c^\k$  becomes the ordinary \emph{Cartan map}
\[
c \colon K_0(A) \to G_0(A)  \ .
\]

A map $T(A) \to T(A)^{*}$ is obtained by sending $a \in A$ to the
linear form $b \mapsto \tr_{A/\k}(L_b \circ R_a)$ on $A$. Here, $R_a, L_b \in \End_\k(A)$ 
denote right and left multiplication by $a$ and  $b$, respectively. Note that if $a$ or
$b$ belongs to $[A,A]$ then $L_b \circ R_a \in [\End_\k(A),\End_\k(A)]$ and so
$\tr_{A/\k}(L_b \circ R_a) = 0$\,. Therefore, $\tr_{A/\k}(L_b \circ R_a)$ depends only on
$T(a)$ and $T(b)$ and we obtain a well-defined $\k$-linear map
\begin{equation} \label{E:projectives}
\,.\,^t \colon T(A) \to T(A)^{*}\ , \qquad T(a) \mapsto (b \mapsto  \tr_{A/\k}(L_b \circ R_a))\ .
\end{equation}

\begin{prop}[Bass \cite{hB76}] \label{P:projectives}
Let $A$ be a $\k$-algebra that is finitely generated projective over $\k$\,. Then the following
diagram commutes:
\[
\xymatrix{%
K_0(A) \ar[r]^{c^\k} \ar[d]_r & G_0^\k(A) \ar[d]^{\chi}\\
T(A) \ar[r]_{\,.\,^t} & T(A)^{*}
} 
\]
\end{prop}


\subsection{Finite-dimensional algebras over a field}  \label{SS:Finite}

Now let $A$ be a finite-dimensional
algebra over a field $\k$ and let $\rad A$ denote the Jacobson radical of $A$.
Since $\rad A$ is nilpotent, the character $\chi_V$ of any $V$ in $\fgmod{A}$
vanishes on $\rad A$, and hence the character map 
$\chi \colon G_0(A)  \to T(A)^{*}$ of Section~\ref{SS:characters} 
actually takes values in $T(A/\rad A)^* \subseteq T(A)^*$.
By $\k$-linear extension of $\chi$ we obtain a map
\[
\til{\chi} \colon G_0(A)\otimes_{\ZZ} \k \tto T(A/\rad A)^* \ .
\]
For the same reason, the map $\,.\,^t$ in \eqref{E:projectives} takes values in $T(A/\rad A)^*$
and it factors through
$T(A/\rad A)$. Hence $\,.\,^t$ factors through a $\k$-linear map 
\[
\til{\,.\,^t} \colon T(A/\rad A) \tto T(A/\rad A)^* \ .
\]
Finally, we have the composite
\[
\til{r} \colon K_0(A) \otimes_{\ZZ}\k \stackrel{r_\k}{\tto} T(A) \stackrel{\text{can.}}{\onto} T(A/\rad A) \\
\]
with $r_\k$ as in \eqref{E:rk}. By Proposition~\ref{P:projectives} these maps fit into a
commutative diagram of $\k$-linear maps
\begin{equation} \label{E:Finite}
\xymatrix{%
K_0(A) \otimes_{\ZZ} \k \ar[r]^{c \otimes 1_\k} \ar[d]_{\til{r}}
& G_0(A)\otimes_{\ZZ} \k \ar[d]^{\til{\chi}}\\
T(A/\rad A) \ar[r]_-{\til{\,.\,^t} } & T(A/\rad A)^*
} 
\end{equation}
where $c$ is the Cartan map from Section~\ref{SS:projectives}. If $\k$ is a \emph{splitting field} for $A$,
that is, $\End_A(V) = \k$ holds for all irreducible $A$-modules $V$,
then both $\til{r}$ and $\til{\chi}$ are isomorphisms; see \cite[1.6]{mL97a}.


\section{Frobenius algebras} \label{S:Frobenius}

In this section, we give an alternative description of the map \eqref{E:projectives} in the special case where
$A$ is a Frobenius algebra over a commutative base ring $\k$ and derive various consequences for the 
Cartan map.


\subsection{Frobenius algebras}  \label{SS:Frobenius}

We briefly recall some basics concerning
Frobenius algebras refering to \cite{mLxx}
for additional information and some details that are omitted below.

The dual $A^* = \Hom_\k(A,\k)$ carries a standard $(A,A)$-bimo\-dule structure:
\begin{equation*} 
(afb)(x) = f(bxa) \qquad (a,b,x \in A, f \in A^*)\ .
\end{equation*}
The $\k$-algebra $A$ is said to be Frobenius if $A$ is finitely generated projective over $\k$ and $A$ is 
isomorphic to $A^*$ as left $A$-module or, equivalently, as right $A$-module.
These isomorphism amount to the existence of a nonsingular associative 
$\k$-bilinear form $\beta \colon A \times A \to \k$. Given such a form $\beta$, we obtain
an isomorphism of left $A$-modules
\begin{equation*} 
\bl_\beta \colon {}_AA \iso {}_AA^*\ , \qquad a \mapsto \beta(\,.\,,a) \ .
\end{equation*}
In place of $\beta$, one can equally well work with the so-called \emph{Frobenius homomorphism}
\begin{equation*} 
\lambda = \lambda_\beta = \beta(\,.\,,1) = \beta(1,\,.\,) \in A^* \ .
\end{equation*}
Indeed, $\beta(a,b) = \lambda(ab)$ holds for all $a,b \in A$. The isomorphism $\bl_\beta$
then takes the form 
\begin{equation} \label{E:Frobenius2}
\bl_\beta(a) = a\lambda\ .
\end{equation}
The linear form $\lambda$ is a free generator of $A^*$ as both left and as right $A$-module;
see, e.g., \cite[1.1.1]{mLxx}. The 
automorphism $\alpha = \alpha_\beta \in \Aut_\text{$\k$-alg}(A)$ that is given by 
\begin{equation} \label{E:Frobenius2'}
\lambda a = \alpha(a)\lambda \qquad (a \in A)
\end{equation} 
is called the \emph{Nakayama 
automorphism} that is associated to $\beta$.


\subsection{Change of bilinear form} \label{SS:change}

If $\beta, \beta' \colon A \times A \to \k$ are two nonsingular associative 
$\k$-bilinear form then the isomorphism
$\bl_\beta^{-1} \circ \bl_{\beta'} \colon  {}_AA \iso {}_AA$
is given by right multiplication by some 
unit $u \in A^\times$.
Hence, $\beta'(\,.\,,\,.\,) = \beta(\,.\,,\,.\,u)$. The Frobenius homomorphisms
$\lambda' = \bl_{\beta'}(1)$ and $\lambda = \bl_\beta(1)$ are
related by $\lambda'= u\lambda$, and the Nakayama automorphisms $\alpha$
and $\alpha'$ that are associated with
$\beta$ and $\beta'$, resp.,  differ by an inner automorphism:
$\alpha'(a) = u \alpha(a)u^{-1}$.


\subsection{Dual bases}  \label{SS:dual}

Let $A$  be a Frobenius $\k$-algebra with nonsingular associative 
$\k$-bilinear form $\beta$ as in \ref{SS:Frobenius}. In view of the canonical isomorphism $\End_\k(A) \cong A \otimes_\k A^*$
the isomorphism 
$\bl_\beta$ yields an isomorphism $\End_\k(A) \iso A \otimes _\k A$.
Writing the image of $1_A \in \End_\k(A)$ under this
isomorphism as $\sum_{i} x_i \otimes y_i \in A \otimes_\k A$, we have
\begin{equation} \label{E:dual}
a = \sum_i \beta(a,y_i) x_i = \sum_i \lambda(ay_i) x_i  \qquad (a \in A)\ .
\end{equation}
The elements $\{ x_i \}, \{ y_i \}$ of $A$ are usually referred to as \emph{dual bases} for $\beta$. 
The first equation in \eqref{E:dual} is equivalent to 
\begin{equation} \label{E:dual'}
a = \sum_i \beta(x_i,a)y_i  \qquad (a \in A)\ ;
\end{equation}
see \cite[equation (8)]{mLxx}.


\subsection{The Higman trace}  \label{SS:Higman}

Let $A$  be a Frobenius $\k$-algebra with nonsingular associative 
$\k$-bilinear form $\beta$ as in \ref{SS:Frobenius}. Since
the element $\sum_{i} x_i \otimes y_i \in A \otimes_\k A$ is completely determined $\beta$,
the map
\begin{equation} \label{E:Higman}
\tau = \tau_\beta \colon A \to A \ , \qquad a \mapsto \sum_i  x_i a y_i
\end{equation}
only depends on $\beta$ and not on the choice of dual bases $\{ x_i \}, \{ y_i \}$ for $\beta$.
Furthermore, $\tau$ is clearly $\cen(A)$-linear, where $\cen(A)$ denotes the center of $A$.

Part (a) of the following lemma gives the desired description of the map \eqref{E:projectives};
part (b) will not be needed in this article but may be of independent interest.

\begin{lem} \label{L:Higman}
Let $(A,\beta)$ be a Frobenius algebra with Frobenius homomorphism 
$\lambda = \lambda_\beta \in A^*$ and Nakayama automorphism $\alpha = \alpha_\beta$,
and let $\tau$ be as in \eqref{E:Higman}. Then:
\begin{enumerate}
\item
$T(a)^t = \tau(a)\lambda$ holds or all $a \in A$.
In particular, $\tau$ vanishes on $[A,A]$. 
\item
$\beta(\tau(a),b) = \beta(a,\tau(b))$ and $a \tau(b) = \tau(b) \alpha(a)$ for all $a,b \in A$.
Moreover, $\alpha\tau = \tau\alpha$.
\end{enumerate}
\end{lem}

\begin{proof}
(a) 
Equation \eqref{E:Hattori} gives
\begin{equation*} 
\tr_{A/\k}(\phi) = \sum_i \beta(\phi(x_i),y_i) =  \sum_i \lambda(\phi(x_i)y_i) 
\end{equation*}
for any $\phi\in\End_\k(A)$.
Applying this to the endomorphism $\phi = L_b \circ R_a$ in \eqref{E:projectives},
we obtain 
\[
\tr_{A/\k}(L_b \circ R_a) = \sum_i  \lambda(bx_i a y_i) = \lambda(b\tau(a))
\]
for $a,b \in A$. This proves the asserted formula for $T(a)^t$.
Since $T([A,A])^t = 0$, it follows that $\tau(\,.\,)\lambda$ vanishes on $[A,A]$. 
Finally, since $\tau(\,.\,)\lambda = \bl_\beta\circ\tau$ by \eqref{E:Frobenius2},
injectivity of $\bl_\beta$ implies that $\tau$ vanishes on $[A,A]$.

(b)
Equation \eqref{E:Frobenius2'} says that
\[
\beta(a,b) = \beta(b,\alpha(a)) \qquad (a,b \in A)\ .
\]
Hence \eqref{E:dual'} can be written as $a = \sum_i \beta(a, \alpha(x_i))y_i$. It follows that
\[
\sum_i x_i \otimes y_i
= \sum_i y_i \otimes \alpha(x_i) \ ,
\]
both elements being the image of 
$1_A \in \End_\k(A)$ under the isomorphism 
$1_A \otimes \bl_\beta^{-1} \colon \End_\k(A) = A \otimes _\k A^* \iso A \otimes _\k A$.
Therefore, 
\begin{equation} \label{E:Higman'}
\tau(a) = \sum_i y_ia \alpha(x_i) \ .
\end{equation}
Now we compute, for $a,b \in A$,
\begin{align*}
\beta(\tau(a),b) &= \sum_i \beta(x_iay_i,b)
= \sum_i \beta(x_ia,y_ib) 
= \sum_i \beta(y_ib, \alpha(x_ia)) \\
&= \sum_i \beta(y_ib\alpha(x_i),\alpha(a)) 
= \sum_i \beta(a, y_ib\alpha(x_i)) 
= \beta(a,\tau(b))\ .
\intertext{and}
a \tau(b) &\underset{\eqref{E:Higman'}}{=} \sum_{i} a y_i b \alpha(x_i)
\underset{\eqref{E:dual'}}{=} \sum_{i,j} \beta(x_j,ay_i) y_j b \alpha(x_i)
= \sum_{i,j} \beta(x_ja,y_i) y_j b \alpha(x_i) \\
&= \sum_{i,j} y_j b \beta(x_ja,y_i) \alpha(x_i)
\underset{\eqref{E:dual}}{=} \sum_j y_j b \alpha(x_j a) 
\underset{\eqref{E:Higman'}}{=} \tau(b) \alpha(a) \ .
\end{align*}
Finally, 
\[
\beta(b, \alpha\tau(a)) = \beta(\tau(a), b) = \beta(a, \tau(b)) = \beta(\tau(b), \alpha(a))
= \beta(b, \tau\alpha(a)) \ ,
\]
which shows that $\alpha\tau(a) = \tau\alpha(a)$.
\end{proof}

We will refer to  $\tau = \tau_\beta$ 
as the \emph{Higman trace} that is associated to $\beta$.
If $\beta' \colon A \times A \to \k$ is another nonsingular associative 
$\k$-bilinear form then 
$\tau_{\beta'}(a) = \tau(a)u^{-1}$ for some 
unit $u \in A^\times$; see \ref{SS:change}.

\subsubsection{} \label{SSS:Casimir} 
In  \cite{dH55}, Higman introduced the following \emph{Casimir operator} 
\footnote{\cite{cCiR62} uses the terminology Gasch{\"u}tz-Ikeda operator.}
for a given nonsingular associative bilinear form $\beta$ on $A$\,:
\begin{equation*} 
A \to \cen(A) \ , \quad a \mapsto \sum_i y_i a x_i \ ;
\end{equation*}
see also \cite[3.1]{mLxx}.
If $A$ is a \emph{symmetric} $\k$-algebra, that is, $A$ and $A^*$ are 
isomorphic as $(A,A)$-bimodules, then the form $\beta$ can be chosen to be
symmetric. The corresponding Nakayama automorphism $\alpha$ is the identity and
by \eqref{E:Higman'}
the Higman trace $\tau$ coincides with the Casimir operator in this case. 

\subsubsection{} \label{SSS:tiltau} 
Let $A$ be a 
Frobenius algebra over a field $\k$. By Section~\ref{SS:Finite}
the map $\,.\,^t$ factors through $T(A/\rad A)$, 
and so Lemma~\ref{L:Higman}(a)
tells us that $\tau$ vanishes on $\rad A$. Moreover, since $\,.\,^t$ 
takes values in $T(A/\rad A)^*$, we also have $b\tau(a)\lambda = T(a)^t(\,.\,b) = 0$
for all $a \in A$, $b \in \rad A$. Hence, $(\rad A)\tau(a) = 0$ and so
$\Im \tau \subseteq \soc A$, the socle of $A$ (which is in fact the same for ${}_AA$
and for $A_A$ by \cite[58.12]{cCiR62}).
This shows that the Higman trace $\tau$ factors through a map
\begin{equation} \label{E:Finite2}
\til{\tau} \colon T(A/\rad A) \tto \soc A \into A \ .
\end{equation}


\subsection{Examples}  \label{SS:Examples}

\subsubsection{} \label{SSS:matrices} 
If $A = \Mat_n(\k)$ is the $n \times n$-matrix algebra over $\k$ then we can take
the ordinary trace $\lambda = \trace \in A^*$ as Frobenius homomorphism.
Dual bases are given by $\sum_i x_i \otimes y_i = \sum_{j,k} e_{j,k}\otimes e_{k,j}$,
where $e_{l,m}$ is the matrix with $1$ in the $(l,m)$-position and $0$s elsewhere.
The Higman trace is 
\[
\tau(a) = \trace(a)1_{n\times n}\ .
\]

\subsubsection{} \label{SSS:groupalgebra} 
The group algebra $A = \k G$ of a finite group $G$ has Frobenius homomorphism 
$\lambda$ with $\lambda(\sum_{g\in G} k_gg) = k_1$.
Dual bases for $\lambda$ are $\sum_i x_i \otimes y_i = \sum_{g\in G} g\otimes g^{-1}$, and
the Higman trace is \[
\tau(a) = \sum_{g\in G} gag^{-1} \ .
\]

\subsubsection{} \label{SSS:Hopfalgebra} 
Generalizing \ref{SSS:groupalgebra}, let $H$ be a Hopf $\k$-algebra that is finitely generated 
projective over $\k$, with augmentation $\e$, antipode $\sS$,
and comultiplication $\Delta$. By \cite{bP71}, $H$ is Frobenius over $\k$ if and only if the $\k$-module 
$\Hint_H^r = \{ a \in H \mid ab =  \e(b) a \}$
of right integrals of $H$ is free of rank $1$ over $\k$. In this case, we may pick
integrals $\Lambda \in \Hint_H^r$ and $\lambda \in \Hint_{H^*}^l$ satisfying 
$\lambda(\Lambda) = 1$.  The form $\lambda$ serves as Frobenius
homomorphism and dual bases are given by 
$\sum x_i \otimes y_i = \sum \Lambda_2 \otimes \sS(\Lambda_1)$. Here 
we have used the standard Sweedler notation 
$\Delta h = \sum h_1 \otimes h_2$.
Thus, the Higman trace takes the form 
\[
\tau(a) = \sum \Lambda_2 a \sS(\Lambda_1)\ .
\]
In the special case where $H = \k G$, we may take $\Lambda = \sum_{g\in G} g$ 
and $\lambda$ as in \ref{SSS:groupalgebra} to obtain the previous formula.


\subsection{Rank and determinant of the Cartan map}  \label{SS:Cartan}

We now specialize to the case where $A$ is a 
Frobenius algebra over a field $\k$. Our goal
is to determine the rank (i.e., the dimension of the image) of the map
\[
c \otimes 1_\k \colon K_0(A) \otimes_{\ZZ} \k \tto G_0(A) \otimes_{\ZZ} \k 
\]
in \eqref{E:Finite}. Since
$K_0(A)$ and $G_0(A)$ are free abelian groups of the same finite rank,
the number of isomorphism classes of irreducible left $A$-modules,
the Cartan map is given by a square integer matrix $C$, called the
\emph{Cartan matrix} of $A$.
The rank of $c \otimes 1_\k$ is equal to the rank of the Cartan
matrix $C$ when $\ch \k = 0$, and to the rank of the reduction of $C$ modulo $p$
in case $\ch \k = p > 0$. 

For part (b) of the following result, recall from \eqref{E:Finite2} that the Higman trace $\tau$ 
factors as $\tau = \text{can}\circ \til{\tau}$ where $\text{can} \colon A \onto T(A/\rad A)$ 
is the canonical map.
Note that $\til{\tau}^{-1}(A^\times \cup \{ 0 \} )$
only depends 
on $A$ and not on the choice of $\beta$. 

\begin{thm} \label{T:Frobenius}
Let $A$ be a Frobenius $\k$-algebra, where $\k$ is a splitting field for $A$ Then:
\begin{enumerate}
\item
$\rank(c \otimes 1_\k) = \rank \tau$.
\item
The following are equivalent:
\begin{enumerate}
\item $A$ is semisimple;
\item $C=\Id$ and $\til{\tau}^{-1}(A^\times \cup \{ 0 \} ) \neq 0$;
\item $\ch \k$ does not divide $\det C$ and $\til{\tau}^{-1}(A^\times \cup \{ 0 \} ) \neq 0$. 
\end{enumerate} 

\end{enumerate} 
\end{thm}

\begin{proof}
(a)
Since $\til{r}$ and $\til{\chi}$ are isomorphisms by our hypothesis
on $\k$, the commutative diagram \eqref{E:Finite}
tells us that $c \otimes 1_\k$  and $\til{\,.\,^t}$ have the same rank.
Finally, by Lemma~\ref{L:Higman}(a) and \eqref{E:Frobenius2}, the rank
of $\til{\,.\,^t}$ is the same as the rank of the Higman trace $\tau$\,.

(b)
For (i) $\Rightarrow$ (ii), note that $C=\Id$ because $\proj{A} = \fgmod{A}$  for
semisimple $A$. Furthermore, $A$ is a finite product of matrix algebras $\Mat_{n_i}(\k)$. 
Letting $a \in A$ denote the element whose $i^\text{th}$ component is the $n_i \times n_i$ diagonal matrix
$\diag(1,0,\dots,0)$ and choosing $\tau$ componentwise as in \ref{SSS:matrices}, we obtain
$\tau(a) = 1$. Therefore, $\til{\tau}^{-1}(A^\times \cup \{ 0 \} ) \neq 0$.

The implication (ii) $\Rightarrow$ (iii) being clear, it remains to prove (iii) $\Rightarrow$ (i).
The hypothesis that $\ch \k$ does not divide $\det C$ says that $c \otimes 1_\k$ is invertible; so we know that
$\rank(c \otimes 1_\k) = \dim_\k G_0(A) \otimes_{\ZZ}\k = 
\dim_\k T(A/\rad A)^*$, where the second equality holds because $\til{\chi}$ is an
isomorphism. Theorem~\ref{T:Frobenius} now gives $\rank \tau = \dim_\k T(A/\rad A)^*$,
and hence the map $\til{\tau}$ in \eqref{E:Finite2} is injective.
Therefore, our hypothesis $\til{\tau}^{-1}(A^\times \cup \{ 0 \} ) \neq 0$ 
implies that there exists $a \in A$ such that $\tau(a)$ is a unit in $A$.
Since $\tau(a) \in \soc A$ by \eqref{E:Finite2}, this forces $A$ to be semisimple.
\end{proof}

Part (b) of Theorem~\ref{T:Frobenius} is most useful for algebras where the condition
$\til{\tau}^{-1}(A^\times \cup \{ 0 \} ) \neq 0$ is \emph{a priori} known to hold. One such 
class of algebras, taken from the original version \cite[Theorem 3.4]{mL97a}
of Theorem~\ref{T:Frobenius}, is as follows.

\subsubsection{} \label{SSS:Hopfalgebra2} 
Let $H$ be a finite-dimensional Hopf algebra over the splitting field $\k$. Assume that $\sS^2$ is inner,
say $\sS^2 = u^{-1}(\,.\,)u$ for some unit $u \in H^\times$. Choosing $\tau$ as in
\ref{SSS:Hopfalgebra} we obtain
\[
u^{-1}\tau(u) = \sum u^{-1}\Lambda_2 u \sS(\Lambda_1) = \sS(\sum \Lambda_1\sS(\Lambda_2))
= \sS(\e(\Lambda)1) \in \k \ .
\]
Since $H$ is augmented, the class of $u$ in $T(H/\rad H)$ is nonzero, and the above
computation shows that $\tau(u) \in \k u \subseteq H^\times \cup \{ 0 \}$. Therefore,
$\til{\tau}^{-1}(H^\times \cup \{ 0 \} ) \neq 0$.


\section{Modules over Hopf comodule algebras} \label{S:Hopf}

In this section, $H$ will denote a 
Hopf algebra over the commutative ring $\k$\,, with
unit $u$, multiplication $m$, counit $\e$, comultiplication $\Delta$, and antipode $\sS$.
We will use the Sweedler notation $\Delta h = \sum h_1 \otimes h_2$.
Throughout, $\otimes = \otimes_\k$ and $\,.\,^* = \Hom_\k(\,.\,,\k)$ denotes $\k$-linear duals.
Finally, $\langle  \,.\, , \,.\, \rangle \colon H^* \times H \to \k$
is the evaluation pairing. 


\subsection{$H$-comodule algebras and $H$-Galois extensions}  \label{SS:Galois}

A $\k$-algebra $B$ is called a right $H$-comodule algebra if $B$ is a right $H$-comodule,
with structure map 
\[
\rho \colon B \rightarrow B \otimes H\ ,\qquad b \mapsto \sum b_0 \otimes b_1\ ,
\]
such that $\rho(ab) = \sum a_0b_0 \otimes a_1b_1$ for all $a,b \in B$ and $\rho(1) = 1 \otimes 1$; see
\cite[4.1.2]{sM93}. The  $H$-coinvariants in $B$, defined by 
\[
A = B^{co\, H} = \{a \in B \mid \rho(a) = a \otimes 1\} \ ,
\]
form a $\k$-subalgebra of $B$.

The standard $(B,B)$-bimodule structure on $B$ gives rise to a $(B,B)$-bimodule structure on $B \otimes H$.
Since $\rho$ is an $(A,A)$-bimodule map, we obtain a natural left $B$-module map
\begin{equation} \label{E:Galois1}
{}'\!\!\rho \colon B \otimes_A B \to B \otimes H\ ,  \qquad b \otimes \til{b} \mapsto (b \otimes 1)\rho(\til{b})
\end{equation}
and a corresponding right $B$-module map
\begin{equation} \label{E:Galois2}
\rho' \colon B \otimes_A B \to B \otimes H\ ,  \qquad b \otimes \til{b} \mapsto \rho(b)(\til{b} \otimes 1)\ .
\end{equation}
The extension $B/A$ is said to be \emph{$H$-Galois} if ${}'\!\!\rho$ or $\rho'$ is bijective. In case
the antipode $\sS$ is bijective both conditions are equivalent: ${}'\!\!\rho$ is bijective if and only
if $\rho'$ is so; see \cite[8.1.1 and subsequent remarks]{sM93}. In particular, this
equivalence holds when $H$ is finitely generated projective over $\k$, because $\sS$ is known to be
bijective in this case \cite[Proposition 4]{bP71}. Furthermore,
if $H$ is finitely generated projective over $\k$ and $B/A$ is $H$-Galois then $B$ is
finitely generated projective as $A$-module both on the left and on the right; see \cite[1.7, 1.8]{hKmT81}.

Important examples of $H$-Galois extensions are
provided by crossed products $B=A\#_{\sigma}H$. In fact, crossed products
are precisely those right $H$-Galois extensions
that enjoy the so-called (right) normal basis property; cf.~\cite[Corollary 8.2.5]{sM93}.


\subsection{Tensor products of modules}  \label{SS:tensor}

Let $B$ be a right $H$-comodule algebra. Given $M$ be in $\Mod{B}$ and $V$ in $\Mod{H}$,
the tensor product $M \otimes V$ becomes a left $B$-module with the ``diagonal'' action
\[
b(m \otimes v) = \sum b_0m \otimes b_1v \qquad (b \in B, m \in M, v \in V)\ .
\]
We note the following basic properties:
\subsubsection{Functoriality} \label{SSS:tensor}
If $f \colon M \to M'$ is a $B$-module map and $g \colon V \to V'$ is an 
$H$-module map, then
$f\otimes g \colon M\otimes V \to M'\otimes V'$ is a $B$-module map.
\subsubsection{Associativity} 
If $V$ and $V'$ are  $H$-modules then, viewing
$V\otimes V'$ as  $H$-module via $\Delta$, we have
$M \otimes(V\otimes V')\cong (M \otimes V)\otimes V'$.
\bigskip

The following lemma describes some special cases of tensor products of $B$-modules
with $H$-modules.

\begin{lem} \label{L:tensor}
Let $B$ be a right $H$-comodule algebra and let $A = B^{co\, H}$. Then:
\begin{enumerate}
\item
Suppose $B/A$ is $H$-Galois, with \eqref{E:Galois2} being bijective. Then,
for any $M$ be in $\Mod{B}$,
\[
M \otimes H \cong B \otimes_A M 
\]
as left $B$-modules.
\item 
Given $L$  in $\Mod{A}$ and $V$ in $\Mod{H}$, there is a $B$-module
isomorphism
\[
(B \otimes_A L)\otimes V \cong B \otimes_A (L \otimes V) \ ,
\]
where $L \otimes V$ is  viewed as $A$-module via $a(l\otimes v) = al \otimes v$.
\end{enumerate}
\end{lem}

\begin{proof}
(a)
Note that $\rho'$ in \eqref{E:Galois2} is actually a $(B,B)$-bimodule map,
where the $(B,B)$-bimodule structure on $B \otimes_A B$ is given by 
${}_BB \otimes_A B_B$ and $B$ acts diagonally on $B \otimes H$ from the left
and via the right regular action on the factor $B$ from the right.
Therefore, we have left $B$-module isomorphisms
\[
B \otimes_A M  \cong (B \otimes_A B) \otimes_B M \cong (B \otimes H) \otimes_B M
\cong M \otimes H
\]
sending $b \otimes m \mapsto \sum b_0m \otimes b_1$\,.

(b)
First construct a map 
$\gamma \colon B \otimes_A (L \otimes V) \to (B \otimes_A L)\otimes V$
as follows. The canonical map $\phi \colon L \to B \otimes_A L$, 
$l \mapsto 1\otimes_A l$, gives rise to a map of $A$-modules,
$\psi =\phi \otimes 1_V \colon L \otimes V \to (B \otimes_A L)\otimes V$. 
Since 
$(B \otimes_A L)\otimes V$ is in fact a $B$-module, with
the diagonal $B$-action, $\psi$ in
turn yields a map of $B$-modules
\begin{alignat}{3} \label{E:tensor1}
\gamma \colon B \otimes_A &(L \otimes V) &\quad &\tto
&\quad &(B \otimes_A L)\otimes V\\
b \otimes_A &(l\otimes v) &&\longmapsto & &{\textstyle\sum} (b_0\otimes_A l)\otimes b_1v \notag
\end{alignat}
For the inverse map, we define
\begin{alignat}{3}  \label{E:tensor2}
\delta \colon (B \otimes_A L)&\otimes V &\quad &\longrightarrow &\quad &B \otimes_A (L \otimes V)  \\
(b \otimes_A l)&\otimes v &&\longmapsto &&{\textstyle\sum} b_0 \otimes_A (l \otimes \sS^{}(b_1)v) \notag
\end{alignat}
To see that this map is well-defined note that the formula is linear in $l$, $v$
and $b$. Moreover, the map is $\k$-balanced, that is,
$\delta((b \otimes_A l)k \otimes v)=\delta((b \otimes_A l)\otimes kv)$
holds for all $k\in \k$. To verify $A$-balancedness, we use the formula
$a\otimes 1=\sum a_0\otimes a_1$. With this, we calculate for $a\in A$
\[
\begin{split}
\delta((ba\otimes_A l) \otimes v) &= \sum b_0a_0 \otimes_A(l\otimes \sS^{}(b_1a_1)v)\\
&= \sum b_0a \otimes_A (l\otimes \sS^{}(b_1)v)
\end{split}
\]
and
\[
\begin{split}
\delta((b \otimes_A al) \otimes v) 
&= \sum b_0 \otimes_A (al\otimes \sS^{}(b_1)v)\\
&= \sum b_0 \otimes_A a(l\otimes \sS^{}(b_1)v)\\
&= \sum b_0a \otimes_A (l\otimes \sS^{}(b_1)v)\ .
\end{split}
\]
It remains to check that $\gamma$ and $\delta$ are indeed inverse to
each other. We will carry out the verification of the identity
$\gamma\circ\delta = 1_{(B \otimes_A L)\otimes V }$; the check of
$\delta\circ\gamma = 1_{B \otimes_A (L \otimes V)}$ can be handled in
an entirely analogous fashion.
\[
\begin{split}
(\gamma\circ\delta)((b \otimes_A l) \otimes v) 
&= \gamma\left( \sum b_0 \otimes_A (l\otimes \sS^{}(b_1)v) \right)\\
&= \sum (b_0 \otimes_A l)\otimes b_1\sS^{}(b_2)v\\
&= \sum (b_0 \otimes_A l)\otimes \langle \e,b_1\rangle v\\
&= (b \otimes_A l) \otimes v\ .
\end{split}
\]
as required. This completes the proof of the lemma.
\end{proof}

\begin{prop} \label{P:tensor}
Let $B$ be a right $H$-comodule algebra.
Given $M$ be in $\Mod{B}$ and $V$ in $\Mod{H}$\,, view 
$M\otimes V$ as $B$-module with the diagonal $B$-action\,.
\begin{enumerate}
\item 
If $M$ is finitely generated and $V$ is finitely generated over $\k$ then 
$M \otimes V$ is finitely generated. 
\item
If $M$ is projective and $V$ is projective over $\k$ then 
$M \otimes V$ is projective.
\end{enumerate}
\end{prop}

\begin{proof}
(a)
Since $M$ is
generated, $M \otimes V$ is a homomorphic image of $B^n\otimes V \cong (B\otimes V)^{n}$
for some $n$. By Lemma \ref{L:tensor}(b) with $L = A$, the $B$-module
$B\otimes V$ with the diagonal $B$-action is isomorphic to $B\otimes V$ with $B$ just acting on the
left factor. Since $V$ is assumed finite over $\k$, it follows that
$B\otimes V$  is finitely generated as $B$-module, and hence so is
$M \otimes V$.

(a)
Since $\,.\, \otimes V$ commutes with direct sums, it suffices to consider the special case 
$M = B$\,. As above,
Lemma \ref{L:tensor}(b) with $L=A$ implies that $B\otimes V$ is projective as $B$-module.
\end{proof}


\subsection{Module structure on Grothendieck groups}  \label{SS:Grothendieck}

Let $B$ be a right $H$-comodule algebra.

\subsubsection{}
The Grothendieck group $G_0^\k(H) = K_0(\fgmod{H}_\k)$ is a ring with multiplication $[V][V'] = [V \otimes V']$,
where $V \otimes V'$ is an $H$-module via $\Delta$. 
By Proposition~\ref{P:tensor}, we may view $K_0(B) = K_0(\proj{B})$ as right module over 
$G_0^\k(H)$:
\begin{equation} \label{E:Grothendieck}
K_0(B) \otimes_\ZZ G_0^\k(H) \tto K_0(B) \ ,\qquad [M] \otimes [V] \mapsto [M \otimes V]\ .
\end{equation}

\subsubsection{}
Similar remarks hold for the Grothendieck group $G_0(B):= K_0(\FPmod{B})$, where 
$\FPmod{B}$ is the full subcategory of $\Mod{B}$ consisting of all left $B$-modules 
of \emph{type FP}$_\infty$, that is, modules $M$ that
have a projective resolution
\[
\dots \to P_n \to P_{n-1} \to \dots \to P_1 \to P_0 \to M \to 0
\]
where all $P_i$ are finitely generated projective $B$-modules. 
If the algebra $B$ is left noetherian then
$\FPmod{B} = \fgmod{B}$, the category of all finitely generated left $B$-modules. 
In general, it is not hard to see that, for any short exact
exact sequence $0 \to M' \to M \to M'' \to 0$ in $\Mod{B}$ , if two of $\{ M',M,M''\}$ belong to
$\FPmod{B}$ then all three do.
Proposition~\ref{P:tensor} implies that, for any $M$ in $\FPmod{B}$ and any $V$ in 
$\fgmod{H}_\k$, the tensor product $M \otimes V$ belongs to $\FPmod{B}$. Therefore,
we also have a $G_0^\k(H)$-module structure on $G_0(B)$:
\begin{equation} \label{E:Grothendieck2}
G_0(B) \otimes_\ZZ G_0^\k(H) \tto G_0(B) \ ,\qquad [M] \otimes [V] \mapsto [M \otimes V]\ .
\end{equation}
The Cartan map $c \colon K_0(B) \to G_0(B)$, coming from the inclusion
$\proj{B} \into \FPmod{B}$, is clearly a $G_0^\k(H)$-module map.

\subsubsection{} \label{SSS:fgproj}
We now turn to the special case where the Hopf algebra $H$ is finitely generated 
projective over $\k$ and $B/A$ is a right $H$-Galois extension (so $A = B^{co\, H}$).
As we remarked in Section~\ref{SS:Galois}, the algebra $B$ is then 
finitely generated projective as left and right $A$-module. Therefore, we have 
well-defined restriction and induction maps
\[
\begin{aligned}
\Res_A^B &\colon K_0(B) \to K_0(A)\ , \qquad & [P] &\mapsto [{}_AP] \\
\Ind_A^B &\colon K_0(A) \to K_0(B)\ ,  & [Q] &\mapsto [B \otimes_A Q]
\end{aligned}
\]
and similarly for $G_0$. Lemma~\ref{L:tensor}(a) has the following 
immediate consequence.

\begin{lem} \label{L:Grothendieck}
Assume that $H$ is finitely generated 
projective over $\k$\,.
Then, for any right $H$-Galois extension $B/A$, 
the map $\Ind_A^B\circ\Res_A^B \colon G_0(B) \to G_0(B)$ 
is given by the action of $[H] \in G_0^\k(H)$ on $G_0(B)$:
\[
(\Ind_A^B\circ\Res_A^B)[M] = [M][H]\ .
\]
Similarly for $K_0(B)$.
\end{lem}


\subsection{The product formula}  \label{SS:product}

For a given right $H$-comodule algebra $B$, we now describe how
the Hattori-Stallings map $r_B \colon K_0(B) \to T(B) = B/[B,B]$ of Section~\ref{SS:Hattori}
behaves with respect to the module action in Section~\ref{SS:Grothendieck}. 
The lemma below generalizes \cite[Prop.~5.5(c)]{hB76} which treats the case where 
$B= H = \k G$ is the group algebra of a finite group $G$.

View $B$ as a left $H^*$-module as in
\cite[1.6.4]{sM93}:
\begin{equation} \label{E:comodule}
f\cdot b = \sum b_0\langle f, b_1 \rangle \qquad (f \in H^*, b \in B)\ .
\end{equation}
The submodule
$[B,B]$ that is spanned by the Lie commutators in $B$ is stable under
the action of the subalgebra $T(H)^* \subseteq H^*$ consisting of all trace forms on $H$:
$f \cdot [b,b'] = \sum [b_0,b'_0]\langle f, b_1b'_1\rangle$ holds for all $f \in T(H)^*$
and $b,b' \in B$. Therefore, $T(B) = B/[B,B]$ becomes a left $T(H)^*$-module
via \eqref{E:comodule}.
It will be convenient to let $H^*$ act from the right on $B$ by defining
\begin{equation} \label{E:comodule2}
b \cdot f:= \sS^*(f)\cdot b = \sum b_0\langle f, \sS(b_1) \rangle \qquad (f \in H^*, b \in B)\ .
\end{equation}
Since $T(H)^*$ is stable under the antipode $\sS^*$ of $H^*$, the group $T(B)$
becomes a right $T(H)^*$-module in this way.

\begin{lem} \label{L:product}
Let $B$ be a right $H$-comodule algebra.
Then, for any $M$ in $\proj{B}$ and $V$ in $\fgmod{H}_\k$, we have
 $r_B(M \otimes V) = r_B(M) \cdot \chi_{V}$\,.
 Thus the following diagram commutes.
 \[
\xymatrix{%
K_0(B) \otimes_\ZZ G_0^\k(H) \ar[r]^-{\eqref{E:Grothendieck}} \ar[d]_{r_B \otimes \chi} & K_0(B) \ar[d]^{r_B}\\
T(B) \otimes_\ZZ T(H)^* \ar[r]_-{\eqref{E:comodule2}} & T(B)
} 
\]
\end{lem}

\begin{proof}
Write $M = e(F)$ with $F = B^r$ free over $B$ and $e = e^2 \in \End_B(F)$. Then 
$M \otimes V$ is the image of $e \otimes 1_V \in \End_B(F\otimes V)$; see Section~\ref{SSS:tensor}.
By Lemma~\ref{L:tensor},
\[
F \otimes V = (B \otimes_A A^r) \otimes V \underset{\delta}{\iso} B \otimes_A (A^r \otimes V)
\]
with $\delta$ as in \eqref{E:tensor2}. Under this isomorphism, 
$e \otimes 1_V \in \End_B(F\otimes V)$ is transformed into 
$e' = \delta \circ (e \otimes 1_V) \circ \gamma \in \End_B(B \otimes_A (A^r \otimes V))$
with $\gamma = \delta^{-1}$ as in \eqref{E:tensor1}.
Fix an $A$-basis $\{ x_i \}_1^r$ of $A^r$ and write $e(1 \otimes_A x_i) = \sum_j e^{i,j} \otimes_A x_j$
with $e^{i,j} \in B$. Then
\[
r_B(M) = \sum_i T(e^{i,i}) \ .
\]
Let $\{v_l,f_l \} \subseteq V \times V^*$ be dual bases for $V$ as $\k$-module; so $1_V = \sum_l v_l \otimes f_l$
under the standard identification $\End_\k(V) = V \otimes V^*$. Then
\[
\begin{split}
e'(1\otimes_A (x_i \otimes v_l)) 
&= \left( \delta \circ (e \otimes 1_V) \circ \gamma \right)(1\otimes_A (x_i \otimes v_l)) \\
&\underset{\eqref{E:tensor1}}{=} \left( \delta \circ (e \otimes 1_V) \right)((1\otimes_A x_i )\otimes v_l) \\
&= \sum_j \delta(( e^{i,j} \otimes_A x_j )\otimes v_l) \\
&\underset{\eqref{E:tensor2}}{=}  \sum_j  e^{i,j}_0 \otimes_A (x_j \otimes \sS(e^{i,j}_1)v_l) \\
&= \sum_{j,k}  e^{i,j}_0 \otimes_A (x_j \otimes \langle f_k, \sS(e^{i,j}_1) v_l \rangle v_k) \\
&= \sum_{j,k}  e^{i,j}_0 \langle f_k, \sS(e^{i,j}_1) v_l \rangle \otimes_A (x_j \otimes v_k) \ .
\end{split}
\]
Therefore,
\[
r_B(M \otimes V) = \sum_{i,l} T(e^{i,i}_0) \langle f_l, \sS(e^{i,i}_1) v_l \rangle
= \sum_{i} T(e^{i,i}_0) \langle \chi_{V},\sS(e^{i,i}_1) \rangle
\underset{\eqref{E:comodule2}}{=} r_B(M) \cdot \chi_{V}\ .
\]
\end{proof}


\subsection{$H$-Galois extensions with $H$ Frobenius}  \label{SS:Frob}

In this section, we focus on right $H$-Galois extensions $B/A$ such that $H$ is 
involutory (i.e., $\sS^2 = 1$) and a
Frobenius algebra over $\k$. By \cite{bP71}, the Frobenius property 
is equivalent to $H$ being finitely generated 
projective over $\k$ and the $\k$-module 
\[
\Hint_H^r = \{ a \in H \mid ab = \langle \e,b \rangle a \}
\]
of right integrals of $H$ being free of rank $1$ over $\k$. 
If $H$ is finitely generated 
projective over $\k$ then $\Hint_H^l = \sS(\Hint_H^r)$, because $\sS$ is an anti-automorphism of $H$,
and $\Hint_H^r$ and $\Hint_H^l$ are finitely generated projective of
constant rank $1$ over $\k$ by \cite[Proposition 3]{bP71}.
In particular, if $\k$ has trivial Picard group then
$H$ is a Frobenius $\k$-algebra if and only if $H$ is finitely generated 
projective over $\k$. 

Define the ideal $\edim H$ of $\k$ by
\[
\edim H := \langle \e, \Hint_H^l \rangle = \langle \e, \Hint_H^r \rangle \ .
\]
For example, if $H = \k G$ is the group algebra of a finite group $G$ 
then $\Hint_H^l  = \Hint_H^r = \k \Lambda$ with  
$\Lambda = \sum_{g\in G} g$ as in \ref{SSS:Hopfalgebra}, and hence $\edim (\k G) = |G| 1$.
In general, by a classical result of Larson and Sweedler \cite{rLmS69} (see also \cite[Corollary 11]{mLxx}),
$H$ is separable over $\k$ if and only if $H$ is Frobenius over $\k$ with $\edim H = \k$\,.

\begin{prop} \label{P:Frob}
Let $B/A$ be a right $H$-Galois extension with $H$ involutory and Frobenius over $\k$
and let $\iota \colon A \into B$ denote the inclusion map. Then the composite
\[
T(\iota)\circ r_A \circ \Res^B_A \colon K_0(B) \to K_0(A) \to T(A) \to T(B)
\]
has image in 
$(\edim H) A \mod [B,B]$.
\end{prop}

\begin{proof}
For any $M$ in $\proj{B}$, we compute 
\[
\begin{split}
(T(\iota)\circ r_A \circ \Res^B_A)(M) 
&\stackrel{\phantom{\text{Lemma}~\ref{L:Grothendieck}}}{\underset{\eqref{E:Hattori2}}{=}} (r_B\circ K_0(\iota)\circ \Res^B_A)(M) \\
&\underset{\text{Lemma}~\ref{L:Grothendieck}}{=} r_B([M][H]) \\
&\underset{\text{Lemma}~\ref{L:product}}{=} r_B(M) \cdot \chi_{H} \ .
\end{split}
\]
Since   $H$ involutory and Frobenius over $\k$, we further know that 
\[
\sS^*(\chi_H) = \chi_H \quad \text{and} \quad
\k \chi_H = (\edim H)\Hint_{H^*}^l \ ;
\]
cf., e.g., \cite[(13) and Lemma 12(a)]{mLxx}.  Moreover, under the left $H^*$-action \eqref{E:comodule}
on $B$, we have $\Hint_{H^*}^l \cdot B \subseteq B^{H^*} = A$; see \cite[1.7.2]{sM93} for the latter equality.
Therefore, the element $r_B(M) \cdot \chi_{H}$ belongs
to $(\edim H)T(\iota)(T(A))$\,.
\end{proof}

Recall from Section~\ref{SSS:commutative} that if $A$ is a commutative ring without idempotents 
$\neq 0, 1$ then $H_0(A) = [\Spec A,\ZZ]$ consists of constant functions. In particular, $\rank Q$ is
constant for every $Q$ in $\proj{A}$; we will denote the constant value of this function by
$\rank_A Q \in \ZZ$.

\begin{thm} \label{T:Frob}
Let $B/A$ be a right $H$-Galois extension such that
$H$ is involutory and Frobenius over $\k$\,. Assume further that
\begin{enumerate}
\renewcommand{\labelenumi}{(\roman{enumi})}
\item
$A$ is commutative without idempotents $\neq 0, 1$, and
\item
there exists 
$f \in T(B)^*$ with $f(1) = 1$\,.
\end{enumerate}
Then $\rank_A M \cdot 1 \in \edim H$ holds for each $M$ in $\proj{B}$.
\end{thm}

\begin{proof}
By Proposition~\ref{P:Frob}, $T(\iota)\circ r_A(M)$ belongs
to $(\edim H)T(B)$ and by hypothesis (i), we have $r_A(M) = \rank_A M \cdot 1 \in A$\,.
Applying the trace function $f \colon T(B) \to \k$ in (ii) to 
$\rank_A M \cdot 1\mod [B,B] \in (\edim H)T(B)$ yields the result.
\end{proof}

In the special case where $B=H$ is a finite-dimensional Hopf algebra over a field $\k$
and $A = \k$, hypothesis
(i) is trivially holds and we can take $f = \e$ 
in (ii). Theorem~\ref{T:Frob}
therefore implies \cite[Theorem 2.3(b)]{mL97a}: 
if $H$ is involutory and not semisimple then $\ch \k$ is positive and a
divisor of $\dim_\k M$ for every $M$ in $\proj{H}$.


\bibliographystyle{amsplain}
\bibliography{../bibliography}

\def\cprime{$'$} \def\cprime{$'$} \def\cprime{$'$} \def\cprime{$'$}
\providecommand{\bysame}{\leavevmode\hbox to3em{\hrulefill}\thinspace}
\providecommand{\MR}{\relax\ifhmode\unskip\space\fi MR }
\providecommand{\MRhref}[2]{%
  \href{http://www.ams.org/mathscinet-getitem?mr=#1}{#2}
}
\providecommand{\href}[2]{#2}
\begin{thebibliography}{10}

\bibitem{hB68}
Hyman Bass, \emph{Algebraic {$K$}-theory}, W. A. Benjamin, Inc., New
  York-Amsterdam, 1968. \MR{40 \#2736}

\bibitem{hB76}
\bysame, \emph{Euler characteristics and characters of discrete groups},
  Invent. Math. \textbf{35} (1976), 155--196. \MR{MR0432781 (55 \#5764)}

\bibitem{kB82}
Kenneth~S. Brown, \emph{Cohomology of groups}, Graduate Texts in Mathematics,
  vol.~87, Springer-Verlag, New York, 1982. \MR{83k:20002}

\bibitem{cCiR62}
Charles~W. Curtis and Irving Reiner, \emph{Representation theory of finite
  groups and associative algebras}, Pure and Applied Mathematics, Vol. XI,
  Interscience Publishers, a division of John Wiley \& Sons, New York-London,
  1962. \MR{MR0144979 (26 \#2519)}

\bibitem{aH65}
Akira Hattori, \emph{Rank element of a projective module}, Nagoya Math. J.
  \textbf{25} (1965), 113--120. \MR{MR0175950 (31 \#226)}

\bibitem{dH55}
Donald~G. Higman, \emph{On orders in separable algebras}, Canad. J. Math.
  \textbf{7} (1955), 509--515. \MR{MR0088486 (19,527a)}

\bibitem{hKmT81}
H.~F. Kreimer and M.~Takeuchi, \emph{Hopf algebras and {G}alois extensions of
  an algebra}, Indiana Univ. Math. J. \textbf{30} (1981), no.~5, 675--692.
  \MR{MR625597 (83h:16015)}

\bibitem{rLmS69}
Richard~G. Larson and Moss~E. Sweedler, \emph{An associative orthogonal
  bilinear form for {H}opf algebras}, Amer. J. Math. \textbf{91} (1969),
  75--94. \MR{MR0240169 (39 \#1523)}

\bibitem{mLxx}
Martin Lorenz, \emph{Some applications of {F}robenius algebras to {H}opf
  algebras}, In: Proceedings of the Coloquio Latinoamericano de {\'A}lgebra,
  Contemporary Mathematics (to appear).

\bibitem{mL97a}
\bysame, \emph{Representations of finite-dimensional {H}opf algebras}, J.
  Algebra \textbf{188} (1997), no.~2, 476--505. \MR{MR1435369 (98i:16039)}

\bibitem{sM93}
Susan Montgomery, \emph{Hopf algebras and their actions on rings}, CBMS
  Regional Conference Series in Mathematics, vol.~82, Published for the
  Conference Board of the Mathematical Sciences, Washington, DC, 1993.
  \MR{94i:16019}

\bibitem{bP71}
Bodo Pareigis, \emph{When {H}opf algebras are {F}robenius algebras}, J. Algebra
  \textbf{18} (1971), 588--596. \MR{MR0280522 (43 \#6242)}

\bibitem{rSeE70}
Richard~G. Swan and E.~Graham Evans, \emph{{$K$}-theory of finite groups and
  orders}, Springer-Verlag, Berlin, 1970, Lecture Notes in Mathematics, Vol.
  149. \MR{46 \#7310}

\bibitem{lT99}
Loretta~F. Tokoly, \emph{Frobenius reciprocity and {G}rothendieck groups of
  {H}opf {G}alois extensions}, Ph.D. thesis, Temple University, 1999.

\bibitem{cWxx}
Charles~A. Weibel, \emph{An introduction to algebraic {K}-theory}, a graduate
  textbook in progress, available at
  \verb+http://www.math.rutgers.edu/~weibel/Kbook.html+.

\end{thebibliography}


\end{document}